\documentclass[12pt]{amsart}
\pdfoutput=1
\usepackage[margin=1in]{geometry}
\usepackage[utf8]{inputenc}
\usepackage{graphicx}
\usepackage{amsmath,amssymb, cite}
\usepackage{amsthm}
\usepackage{color}
\usepackage{tikz}
\usepackage{pgfplots}
\usepackage{url}
\usepackage{mathrsfs}
\allowdisplaybreaks

\makeatletter
\def\l@section{\@tocline{1}{12pt plus2pt}{0pt}{}{\bfseries}}
\def\l@subsection{\@tocline{2}{0pt}{2pc}{2pc}{}}
\makeatother

\makeatletter
\def\subsection{\@startsection{subsection}{2}{\z@}%
	{-3.25ex\@plus -1ex \@minus -.2ex}%
	{1.5ex \@plus .2ex}%
	{\normalfont\bfseries\boldmath}}
\def\subsubsection{\@startsection{subsubsection}{3}%
	\z@{.5\linespacing\@plus.7\linespacing}{-.5em}%
	{\normalfont\bfseries\boldmath}}
\renewcommand\paragraph{\@startsection{paragraph}{4}{\z@}%
	{3.25ex \@plus1ex \@minus.2ex}%
	{-1em}%
	{\normalfont\normalsize\bfseries}}
\makeatother

\theoremstyle{plain}
\newtheorem{thm}{Theorem}[section]

\newtheorem{cor}[thm]{Corollary}
\newtheorem{lem}[thm]{Lemma}

\theoremstyle{definition}

\theoremstyle{remark}
\newtheorem{rem}[thm]{Remark}

\theoremstyle{plain}

\numberwithin{equation}{section}

\newtheorem{opques}[thm]{Open Question}

\newcommand{\thistheoremname}{}
\newtheorem{genericthm}[thm]{\thistheoremname}

\newtheorem*{genericthm*}{\thistheoremname}
\newenvironment{namedthm*}[1]
  {\renewcommand{\thistheoremname}{#1}%
   \begin{genericthm*}}
  {\end{genericthm*}}

\newcommand{\D}{{\mathbb D}}

\newcommand{\C}{{\mathbb C}}

\newcommand{\T}{{\mathbb T}}

\newcommand{\calQ}{{\mathcal Q}}

\newcommand{\HT}{\mathcal{T}}

\newcommand{\HD}{\mathcal{D}}

\def\one{\mbox{1\hspace{-4.25pt}\fontsize{12}{14.4}\selectfont\textrm{1}}}

\begin{document}

\title[]{On the holomorphic differential operator $\frac{d}{dz}: Q_K\to L^q(WdA)$ }

\author{Bingyang Hu}
\address{(Bingyang Hu) Department of Mathematics and Statistics\\
        Auburn University\\
        Auburn, Alabama, U.S.A, 36849}
\email{bzh0108@auburn.edu}

\author{Shuaibing Luo}
\address{(Shuaibing Luo) School of Mathematics\\
        Hunan University\\
        Changsha, Hunan, PR China, 410082}
\email{sluo@hnu.edu.cn}

\author{Jie Xiao}
\address{(Jie Xiao) Department of Mathematics and Statistics\\
        Memorial University of Newfoundland\\
        St. John's, NL, Canada, A1C 5S7}
\email{jxiao@mun.ca}

\author{Xiaojing Zhou}
\address{(Xiaojing Zhou) Department of Mathematics and Statistics\\
         Auburn University\\
         Auburn, Alabama, U.S.A, 36849}
\email{xiz0003@auburn.edu}

\begin{abstract}
In this paper, we obtain
non-testing characterizations, in terms of dyadic capacity gauges, of the
boundedness and compactness of the differentiation operator
$$
\frac{d}{dz}:Q_K\longrightarrow L^q(W\,dA),
\qquad 0<q<\infty.
$$
We also characterize the limiting case as $q\to0^+$, formulated in terms of
a logarithmic geometric mean, while the endpoint $q=\infty$ is treated separately using a
standard testing argument. These results greatly extend the previous work on
$\calQ_p$-spaces to the general setting of $Q_K$-spaces.
As applications, we characterize composition
operators and Volterra-type integral operators between different
$Q_K$-spaces. In particular, the off-diagonal characterization established
here, together with the previously established diagonal case, completely
resolves Zhao's 2009 open question on composition operators between
$\calQ_p$-spaces.
\end{abstract}

\date{\today}



\maketitle
\tableofcontents

\arraycolsep=1pt
	
\section{Introduction}\label{s1}

The \emph{goal} of this paper is to further the line of research initiated by the first and fourth authors in \cite{HZ26}, where they obtained non-testing characterizations of bounded and compact composition operators on $\calQ_p$-spaces using ideas from dyadic harmonic analysis. We aim to extend this dyadic framework to study boundedness and compactness of the more general weighted differentiation operator
\begin{equation} \label{20260726prob01}
\frac{d}{dz}:Q_K\longrightarrow L^q(WdA),
\end{equation}
where\footnote{Here, the case $q=0$ is understood in the limiting sense as $q\to0^+$; see Theorem~\ref{20260728thm02} for a precise formulation.} $0 \le q \le \infty$, $W\geq 0$ is a locally integrable function (weight) on the open unit disc $\D$, $dA$ is the normalized area measure on $\D$, and $L^q(WdA)$ is the weighted $L^p$-space induced by the measure $WdA$. As applications, we are able to
\begin{enumerate}
    \item[$\bullet$] completely characterize the boundedness and compactness of composition operators acting on various $Q_K$-spaces in terms of the corresponding dyadic gauge conditions (see, Corollary \ref{c2}), and therefore fully resolve the off-diagonal case of Zhao's 2009 open problem \cite{Zhao09} (see, Corollary \ref{cZhao});

    \item [$\bullet$] completely characterize the boundedness and compactness of Volterra-type integral operators acting on various $Q_K$-spaces, in terms of the corresponding dyadic gauge conditions (see, Corollary \ref{c3}).
\end{enumerate}

We begin by recalling the definition of the $Q_K$-spaces, which may be viewed as weighted generalizations of the $\calQ_p$-spaces; see the monograph of Wulan and Zhu \cite{WZ17} for further background. Let $K$ be a non-decreasing and right-continuous self-map of $[0,\infty)$  obeying the following conditions:
\begin{enumerate}
\item[(a)] \emph{Positivity.}
$$
K(0)=0
\qquad\text{and}\qquad
K(t)>0 \quad \text{for all } t>0.
$$

\item[(b)] \emph{Integrability.}
$$
\int_{\D}K\left(\log\frac{1}{|z|}\right)\,dA(z)<\infty.
$$

\item[(c)] \emph{Size.} Define
$$
\varphi_K(s):=\sup_{0<t\leq 1}\frac{K(st)}{K(t)},
\qquad s>0.
$$
Then
$$
\int_0^1\frac{\varphi_K(s)}{s}\,ds<\infty,
$$
and there exists some $\sigma>0$ such that
$$
\int_1^\infty\frac{\varphi_K(s)}{s^{1+\sigma}}\,ds<\infty.
$$

\item[(d)] \emph{Doubling.} There exists a constant $C_K>0$ such that
$$
K(2t)\leq C_KK(t),
\qquad 0<t\leq\frac{1}{2}.
$$
\end{enumerate}
Important examples of such weights include the following:
\begin{enumerate}
\item[$\bullet$] The power weights
$$
K_p(t)=t^p,
\qquad
(p,t)\in(0,\infty)\times[0,1].
$$

\vspace{0.1cm}

\item[$\bullet$] The log--bump weights
$$
K_{\log,p}(t)
=
t\left(\log\frac{e^p}{t}\right)^p,
\qquad
(p,t)\in[0,\infty)\times[0,1].
$$
\end{enumerate}
These two choices recover the spaces $\calQ_p$ and
$\calQ_{\log,p}$, respectively; see, for example,
\cite{Xi01,Xi06} and \cite{LX20,Xi00}.

\vspace{0.1cm}

Let $Q_K$ denote the holomorphic M\"obius-invariant space on $\D$ defined by
$$
\left\{
\begin{aligned}
&Q_K
:=\left\{f\in\text{Hol}(\D):\|f\|_{Q_K,*}<\infty\right\},\\
&\|f\|_{Q_K}:=|f(0)|+\|f\|_{Q_K,*},\\
&\|f\|_{Q_K,*}^2 :=\sup_{a\in\D}\int_{\D}\left|(f\circ\varphi_a)'(z)\right|^2
K\left(1-|z|^2\right)\,dA(z) \simeq\sup_{J\subseteq\T}\int_{Q_J}|f'(z)|^2
K\left(\frac{1-|z|}{|J|}\right)\,dA(z), \\
& \qquad \qquad \quad \text{provided that $K$ satisfies conditions \emph{(a)--(d)}.}
\end{aligned}
\right.
$$
Here, $\text{Hol}(\D)$ denotes the space of  all holomorphic functions on $\D$ equipped with compact--open topology,
$\varphi_a(z)=\frac{a-z}{1-\bar{a}z}, a \in \D $ is the disc automorphism,
$J\subseteq\T = \partial \D$ is an arc and $Q_J$ denotes the usual Carleson tent
\begin{align*}
Q_J := \bigl\{ r e^{it} : e^{it} \in J,\ 1 - |J| \le r < 1 \bigr\},
\end{align*}
see \cite{EWX06,WZ17}. Consequently, $\{K_p,K_{\log,p}\}$ generates the known holomorphic M\"{o}bius invariant space pair $\{\calQ_p, \calQ_{\log,p}\}$.

\vspace{0.1cm}

We first treat the main case, namely, $0<q<\infty$. Following the dyadic framework developed in \cite{HZ26} and the related work \cite{HZ26b,FXX}, we expect that the behavior of the holomorphic differentiation operator $d/dz$ in \eqref{20260726prob01} can be characterized in terms of an appropriate dyadic gauge. For this purpose, let $\HD$ be a dyadic system on $\T=\partial\D$ and let $h=\{h(I)\}_{I\in\HD}$ be a nonnegative sequence. We define the \emph{$Q_K$-capacity gauge} of $h$ by
\begin{align} \label{20260726cond01}
\|h\|_{B_K^q(\HD)}
:=\sup\left\{
\sum_{I\in\HD}d_I^{q/2}h(I):
d_I\geq0\ \text{and}\
\sup_{J\in\HD}\sum_{I\subseteq J}
\frac{K\left(\frac{|I|}{|J|}\right)}{K(|I|)}d_I\leq1
\right\},
\end{align}
where $0 < q < \infty$.
This expression shows up when we try to study the boundedness and compactness of the differentiation operator
$$
\frac{d}{dz}:Q_K\longrightarrow L^q(W\,dA).
$$

For $I \in \HD$, let
$$Q_I^{\text{up}} := \left\{ r e^{it} : e^{it} \in I,\ 1 - |I| \le r < 1 - \frac{|I|}{2} \right\}$$
be the upper part of the Carleson tent. If $E$ is a set, then we let $\one_E$ denote the characteristic function of the set $E$.
Our first main result is the following.

\begin{thm}\label{t1}
Let
$$
\begin{cases}
(q,\rho)\in(0,\infty)\times(0,1);\\
Q_I^{\mathrm{up}}=\text{the upper half part of the Carleson tent $Q_I$};\\
h_{W,q}^K(I):=\dfrac{\displaystyle\int_{Q_I^{\mathrm{up}}}W(z)\,dA(z)}
{\bigl(|I|^2K(|I|)\bigr)^{q/2}};\\
h_{W_\rho,q}^K(I):=\dfrac{\displaystyle\int_{Q_I^{\mathrm{up}}}\one_{\{|z|>\rho\}}W(z)\,dA(z)}
{\bigl(|I|^2 K(|I|)\bigr)^{q/2}}.
\end{cases}
$$
Then the following assertions hold.
\begin{itemize}
\item[(i)] The differentiation operator
$$
\frac{d}{dz}:Q_K\longrightarrow L^q(W\,dA)
$$
is bounded, that is, there exists a constant $C>0$ such that
\begin{align*}
\int_\D |f'(z)|^q W(z) dA(z) \leq C \|f\|_{Q_{K,*}}^q, f \in Q_K,
\end{align*}
if and only if there exists a dyadic system $\HD$ on $\T$ such that
$$
\bigl\|h_{W,q}^K\bigr\|_{B_K^q(\HD)}<\infty.
$$
\item[(ii)] The differentiation operator
$$
\frac{d}{dz}:Q_K\longrightarrow L^q(W\,dA)
$$
is compact if and only if
$$
\begin{cases}
\dfrac{d}{dz}:Q_K\longrightarrow L^q(W\,dA)\ \text{is bounded};\\
\displaystyle\lim_{\rho\to1^-}\sup_{\HD}
\bigl\|h_{W_\rho,q}^K\bigr\|_{B_K^q(\HD)}=0,
\end{cases}
$$
where the supremum is taken over all dyadic systems $\HD$ on $\T$.
\end{itemize}
\end{thm}

\begin{rem}
\begin{enumerate}
\item Theorem~\ref{t1} extends both the dyadic trace theorem
\cite[Theorem~1.4]{HZ26} and its compact counterpart
\cite[Theorem~5.2]{HZ26} to the more general setting of $Q_K$-spaces
and arbitrary exponents $0<q<\infty$. The original results
concern the holomorphic differentiation operator
$$
\frac{d}{dz}:\calQ_p\longrightarrow L^2(W\,dA).
$$
The proof of Theorem~\ref{t1} is also much more involved and some new ingredients are introduced.

\item Theorem~\ref{t1} also shows that the dyadic framework initiated
in \cite{HZ26} and further developed here provides a unified approach
to embedding problems in the holomorphic setting. To illustrate the unified nature of Theorem~\ref{t1}, let us briefly
recall a classical example. For $1\leq p,q<\infty$, consider the embedding
$$
\operatorname{id}:A^p(\D)\longrightarrow L^q(\mu),
$$
where $A^p(\D)$ denotes the standard Bergman space. Its boundedness
characterization naturally splits into the two cases $p\leq q$ and $p>q$.
More precisely, fix $0<r<1$, let $\Delta(z,r)$ denote the
pseudohyperbolic disk centered at $z$ with radius $r$, and define $\widehat{\mu}_r(z):= \frac{\mu\bigl(\Delta(z,r)\bigr)} {(1-|z|^2)^2}$. Then the above embedding is bounded if and only if
$$
\begin{cases}
\displaystyle
\sup_{z\in\D}
\frac{\mu\bigl(\Delta(z,r)\bigr)}
     {(1-|z|^2)^{2q/p}}
<\infty,
& p\leq q,\\[3ex]
\displaystyle
\widehat{\mu}_r\in L^{p/(p-q)}(\D,dA),
& p>q.
\end{cases}
$$
We refer the reader to the influential works of Luecking
\cite{Luecking1983,Luecking1991,Luecking1993} for further details.

By contrast, although $Q_K$ is defined through an $L^2$-integral of
$f'$, the $Q_K$-capacity gauge in \eqref{20260726cond01} takes the
same form for $q\leq2$ and $q>2$, and yields a unified characterization
for all $0<q<\infty$.
\end{enumerate}
\end{rem}

For the endpoint case $q=\infty$, we use the standard testing argument to obtain the following result.

\begin{thm}\label{20260726thm01}
Let $K$ satisfy assumptions {\rm (a)--(d)}, and let $W$ be a weight on $\D$. Then the following assertions are equivalent:
\begin{itemize}
\item[(i)] The differentiation operator
$$
\frac{d}{dz}:Q_K\longrightarrow L^\infty(W\,dA)
$$
is bounded;

\item[(ii)] The differentiation operator
$$
\frac{d}{dz}:Q_K\longrightarrow L^\infty(W\,dA)
$$
is compact;

\item[(iii)] There exists $0<r<1$ such that
$$
W(z)=0
$$
for almost every $z\in\D$ with $|z|>r$.
\end{itemize}
\end{thm}

For the limiting endpoint case corresponding to $q=0$, we first explain what should replace $L^q(WdA)$ at this endpoint. Let $W$ be a nonzero weight satisfying $0<W(\D):=\int_{\D}W(z)\,dA(z)<\infty$. For $0<q<\infty$ and a measurable function $g$ on $\D$, define
$$
M_{q,W}(g):=\left(\frac{1}{W(\D)}
\int_{\D}|g(z)|^qW(z)\,dA(z)\right)^{1/q}.
$$
Suppose that $\log|g|$ is integrable with respect to $WdA$ and that
$$
\int_{\D}|g(z)|^{q_0}W(z)\,dA(z)<\infty
$$
for some $q_0>0$. Since
\begin{equation} \label{20260728eq38}
|g|^q=\exp\bigl(q\log|g|\bigr)=1+q\log|g|+o(q)
\end{equation}
in $L^1(WdA)$ for every $q>0$ sufficiently small, we obtain
\begin{equation} \label{20260728eq39}
\frac{1}{W(\D)}\int_{\D}|g(z)|^qW(z)\,dA(z)=
1+\frac{q}{W(\D)}
\int_{\D}W(z)\log|g(z)|\,dA(z)+o(q).
\end{equation}
It follows that
\begin{equation} \label{20260728eq40}
\lim_{q\to0^+}M_{q,W}(g)=\exp\left(
\frac{1}{W(\D)}
\int_{\D}W(z)\log|g(z)|\,dA(z)
\right).
\end{equation}
This motivates defining
$$
M_{0,W}(g):=
\exp\left(
\frac{1}{W(\D)}
\int_{\D}W(z)\log|g(z)|\,dA(z)
\right),
$$
at the endpoint $q=0$, which leads to the study of the following endpoint problem.

\begin{thm}\label{20260728thm02}
Let $K$ satisfy assumptions {\rm (a)--(d)}, and let $W$ be a nonzero
weight on $\D$ such that $0<W(\D)<\infty$. Then the following statements
are equivalent:

\begin{itemize}
\item[(i)] There exists a constant $C>0$ such that
$$
M_{0,W}(f')\leq C\|f\|_{Q_{K,*}},
\qquad f\in Q_K.
$$

\item[(ii)] There exists a constant $C>0$ such that
$$
M_{0,W}(f')\leq C\|f\|_{Q_{K,*}},
\qquad f\in Q_{K,0},
$$
where $Q_{K,0}$ denotes the closure of the analytic polynomials in the
$Q_K$-norm.

\item[(iii)] The weight $W$ satisfies
$$
\int_{\D}W(z)\log\frac{e}{1-|z|^2}\,dA(z)<\infty.
$$
\end{itemize}
\end{thm}

\begin{rem}
The proof of Theorem~\ref{20260728thm02} is of independent interest. Although neither of the two conditions in Theorem~\ref{20260728thm02} involves a dyadic gauge, the proof uses a dyadic approach similar to those employed in the proof of Theorem~\ref{t1} and \cite[Theorem~1.4]{HZ26}.
\end{rem}

\medskip

We next study several applications of our main results.

\subsection{Application I: Characterization of
$C_\varphi:Q_{K_1}\to Q_{K_2}$ and a solution to Zhao's off-diagonal problem}

As an immediate consequence of Theorem~\ref{t1}, we obtain a complete
characterization of the boundedness and compactness of the composition
operator induced by a holomorphic self-map $\varphi$ of $\D$:
$$
C_\varphi:Q_{K_1}\longrightarrow Q_{K_2},
\qquad
C_\varphi f:=f\circ\varphi.
$$
We have the following.

\begin{cor}\label{c2}
Let $K_1$ and $K_2$ be two weights satisfying the assumptions (a)--(d) imposed
on $K$, and let
$$
\begin{cases}
\varphi=\text{a holomorphic self-map of $\D$};\\[2mm]
N_{K_2,\varphi,a}(\zeta)
:=\displaystyle\sum_{\varphi(z)=\zeta}
K_2\left(1-|\varphi_a(z)|^2\right),
\quad a\in\D;\\[3mm]
h_{K_1,K_2,a}^{\varphi}(I)
:=\displaystyle\frac{1}{|I|^2K_1(|I|)}
\int_{Q_I^{\mathrm{up}}}
N_{K_2,\varphi,a}(\zeta)\,dA(\zeta),
\quad a\in\D,\ I\subseteq\T;\\[4mm]
H_{a,\rho}^{K_1,K_2,\varphi}(I)
:=\displaystyle\frac{1}{|I|^2K_1(|I|)}
\int_{Q_I^{\mathrm{up}}\cap\{\rho<|\zeta|<1\}}
N_{K_2,\varphi,a}(\zeta)\,dA(\zeta),\\
\hfill (a,\rho)\in\D\times(0,1),\ I\subseteq\T.
\end{cases}
$$
Then the following assertions hold.
\begin{itemize}
\item[(i)] The composition operator
$$
C_\varphi:Q_{K_1}\longrightarrow Q_{K_2}
$$
is bounded if and only if
$$
\exists\ \text{a dyadic system $\HD$ on $\T$ such that}\quad
\sup_{a\in\D}
\bigl\|h_{K_1,K_2,a}^{\varphi}\bigr\|_{B_{K_1}^2(\HD)}
<\infty.
$$

\item[(ii)] The composition operator
$$
C_\varphi:Q_{K_1}\longrightarrow Q_{K_2}
$$
is compact if and only if
$$
\begin{cases}
C_\varphi:Q_{K_1}\longrightarrow Q_{K_2}
\ \text{is bounded};\\[2mm]
\displaystyle
\lim_{\rho\to1^-}\sup_{a\in\D}\sup_{\HD}
\bigl\|H_{a,\rho}^{K_1,K_2,\varphi}\bigr\|_{B_{K_1}^2(\HD)}
=0,
\end{cases}
$$
where the supremum $\sup_{\HD}$ is taken over all dyadic systems
$\HD$ on $\T$.
\end{itemize}
\end{cor}

\begin{proof}
The result follows from an appropriate modification of the arguments used in
\cite[Theorems~1.2 and~1.7]{HZ26}, together with the following
change-of-variables formula:
$$
\begin{aligned}
&\int_{\D}|(f\circ\varphi)'(z)|^2
K_2\left(1-|\varphi_a(z)|^2\right)dA(z)\\
&\qquad=
\int_{\D}|f'(\zeta)|^2
\sum_{\varphi(z)=\zeta}
K_2\left(1-|\varphi_a(z)|^2\right)dA(\zeta)\\
&\qquad=
\int_{\D}|f'(\zeta)|^2
N_{K_2,\varphi,a}(\zeta)\,dA(\zeta).
\end{aligned}
$$
The remaining details are omitted.
\end{proof}

We now explain how Corollary~\ref{c2} resolves the off-diagonal part
of Zhao's open problem. Recall that Zhao \cite{Zhao09} posed the
following question concerning composition operators acting on
$\calQ_p$-spaces.

\begin{opques}
What are the conditions in terms of $\varphi$ so that $C_\varphi$ is
bounded or compact on $\calQ_p$, or between different $\calQ_p$-spaces,
for $0<p<1$?
\end{opques}

In \cite{HZ26}, the first and fourth authors considered only the
diagonal part of this problem, namely,
$$
C_\varphi:\calQ_p\longrightarrow\calQ_p;
$$
see, \cite[Theorems~1.2 and~1.7]{HZ26}. Corollary~\ref{c2} shows that the dyadic trace theorem developed there,
as well as its more general form given by Theorem~\ref{t1}, also yields
a complete solution to the off-diagonal part of Zhao's question.

More precisely, we have the following.

\begin{cor}\label{cZhao}
Let
$$
\begin{cases}
(p_1,p_2,\rho)\in(0,1)^2\times(0,1),
\quad p_1\neq p_2;\\[1mm]
\varphi=\text{a holomorphic self-map of $\D$};\\[2mm]
N_{p_2,\varphi,a}(\zeta)
:=\displaystyle\sum_{\varphi(z)=\zeta}
\left(1-|\varphi_a(z)|^2\right)^{p_2},
\quad a\in\D;\\[3mm]
h_{p_1,p_2,a}^{\varphi}(I)
:=\displaystyle\frac{1}{|I|^{p_1+2}}
\int_{Q_I^{\mathrm{up}}}
N_{p_2,\varphi,a}(\zeta)\,dA(\zeta),
\quad a\in\D,\ I\subseteq\T;\\[4mm]
H_{a,\rho}^{p_1,p_2,\varphi}(I)
:=\displaystyle\frac{1}{|I|^{p_1+2}}
\int_{Q_I^{\mathrm{up}}\cap\{\rho<|\zeta|<1\}}
N_{p_2,\varphi,a}(\zeta)\,dA(\zeta),\\
\hfill (a,\rho)\in\D\times(0,1),\ I\subseteq\T.
\end{cases}
$$
Then the following assertions hold.
\begin{itemize}
\item[(i)] The composition operator
$$
C_\varphi:\calQ_{p_1}\longrightarrow\calQ_{p_2}
$$
is bounded if and only if
$$
\exists\ \text{a dyadic system $\HD$ on $\T$ such that}\quad
\sup_{a\in\D}
\bigl\|h_{p_1,p_2,a}^{\varphi}\bigr\|_{B_{K_{p_1}}^2(\HD)}
<\infty.
$$

\item[(ii)] The composition operator
$$
C_\varphi:\calQ_{p_1}\longrightarrow\calQ_{p_2}
$$
is compact if and only if
$$
\begin{cases}
C_\varphi:\calQ_{p_1}\longrightarrow\calQ_{p_2}
\ \text{is bounded};\\[2mm]
\displaystyle
\lim_{\rho\to1^-}\sup_{a\in\D}\sup_{\HD}
\bigl\|H_{a,\rho}^{p_1,p_2,\varphi}\bigr\|_
{B_{K_{p_1}}^2(\HD)}
=0,
\end{cases}
$$
where the supremum $\sup_{\HD}$ is taken over all dyadic systems
$\HD$ on $\T$.
\end{itemize}
\end{cor}

\begin{proof}
Apply Corollary~\ref{c2} with $K_1=K_{p_1}$ and $K_2=K_{p_2}$,
where $K_{p_j}(t)=t^{p_j}$. Since $Q_{K_{p_j}}=\calQ_{p_j}$ and
$|I|^2K_{p_1}(|I|)=|I|^{p_1+2}$, both assertions follow immediately.
\end{proof}

\subsection{Application II: Characterization of
Volterra-type integral operators between $Q_K$-spaces}

As a second immediate consequence of Theorem~\ref{t1}, we obtain a
complete characterization of the boundedness and compactness of
Volterra-type integral operators between $Q_K$-spaces.

\begin{cor}\label{c3}
Let
$$
\begin{cases}
K_1\text{ and }K_2\text{ be two weights satisfying the assumptions (a)--(d)
imposed on }K;\\[1mm]
(g,\rho)\in\operatorname{Hol}(\D)\times(0,1);\\[1mm]
\displaystyle
I_gf(z):=\int_0^z g(\zeta)f'(\zeta)\,d\zeta,
\quad (f,z)\in Q_{K_1}\times\D;\\[3mm]
W_{g,K_2}(a,z)
:=
|g(z)|^2K_2\left(1-|\varphi_a(z)|^2\right),
\quad (a,z)\in\D\times\D.
\end{cases}
$$
Then the following assertions hold.
\begin{itemize}
\item[(i)] The Volterra-type integral operator
$$
I_g:Q_{K_1}\longrightarrow Q_{K_2}
$$
is bounded if and only if
$$
\exists\ \text{a dyadic system $\HD$ on $\T$ such that}\quad
\sup_{a\in\D}
\bigl\|
h_{W_{g,K_2}(a,\cdot),2}^{K_1}
\bigr\|_{B_{K_1}^2(\HD)}
<\infty.
$$

\item[(ii)] The Volterra-type integral operator
$$
I_g:Q_{K_1}\longrightarrow Q_{K_2}
$$
is compact if and only if
$$
\begin{cases}
I_g:Q_{K_1}\longrightarrow Q_{K_2}
\text{ is bounded};\\[2mm]
\displaystyle
\lim_{\rho\to1^-}
\sup_{a\in\D}\sup_{\HD}
\bigl\|
h_{(W_{g,K_2}(a,\cdot))_\rho,2}^{K_1}
\bigr\|_{B_{K_1}^2(\HD)}
=0,
\end{cases}
$$
where the supremum $\sup_{\HD}$ is taken over all dyadic systems
$\HD$ on $\T$.
\end{itemize}
\end{cor}

\begin{proof}
Since
$$
I_gf(0)=0
\qquad\text{and}\qquad
(I_gf)'(z)=g(z)f'(z),
$$
we have
$$
\begin{aligned}
\|I_gf\|_{Q_{K_2}}^2
&=
\sup_{a\in\D}
\int_{\D}
|f'(z)|^2|g(z)|^2
K_2\left(1-|\varphi_a(z)|^2\right)dA(z)\\
&=
\sup_{a\in\D}
\int_{\D}|f'(z)|^2
W_{g,K_2}(a,z)\,dA(z).
\end{aligned}
$$
Both assertions now follow from Theorem~\ref{t1}.
\end{proof}

The remainder of the paper is organized as follows. The boundedness and compactness assertions in Theorem~\ref{t1} are proved in Sections~\ref{s2} and~\ref{s3}, respectively, while the endpoint results, Theorems~\ref{20260726thm01} and~\ref{20260728thm02}, are established in Sections~\ref{sec04} and~\ref{Sec05}, respectively. \\
\noindent {\bf Acknowledgement.} The first author was supported by NSF grant DMS-2555999 and by the Simons Travel grant MPS-TSM-00007213. The second author was supported by NNSFC (12271149), Natural Science Foundation of Hunan Province (2024JJ2008). The third author was supported by NSERC of Canada \# 202979.

\section{Proof of Theorem~\ref{t1}: Boundedness}\label{s2}

We divide the proof of Theorem~\ref{t1} into two parts.

\subsection{Proof of Theorem~\ref{t1}: Sufficiency for boundedness}

Let
$$
\begin{cases}
f \in Q_K;\\
d_{I,K}(f) = \sup_{z\in Q_I^{\text{up}}} \big|f'(z)\big|^2 (|I|^2 K(|I|)).
\end{cases}
$$
For any $J \in \HD$, we have
\begin{align*}
\sum_{I \subseteq J}\frac{K\left(\frac{|I|}{|J|}\right)}{K(|I|)}d_{I,K}(f)
& = \sum_{I \subseteq J}\frac{K\left(\frac{|I|}{|J|}\right)}{K(|I|)} \sup_{z\in Q_I^{\text{up}}} \big|f'(z)\big|^2 (|I|^2 K(|I|))\\
& = \sum_{I \subseteq J} |I|^2 K\left(\frac{|I|}{|J|}\right) \sup_{z\in Q_I^{\text{up}}} \big|f'(z)\big|^2\\
& \lesssim  \sum_{I \subseteq J} K\left(\frac{|I|}{|J|}\right) \int_{\widetilde{Q_I^{up}}} |f'(z)|^2 dA(z),
\end{align*}
where in the last inequality we used the sub-mean value property for holomorphic functions, and $\widetilde{Q_I^{up}}$ denotes a fixed slight enlargement of $Q_I^{up}$ such that
$$\widetilde{Q_I^{up}} = 1.1 Q_I^{up}.$$

Note that
$$1 - |z| \simeq |I|~~\ \forall\ z \in \widetilde{Q_I^{up}}.$$ Thus, since $K$ is non-decreasing and doubling, we have
\begin{align*}
K\left(\frac{|I|}{|J|}\right) \simeq K\left(\frac{1-|z|}{|J|}\right).
\end{align*}
Since $\{\widetilde{Q_I^{up}}\}_{I \in \HD}$ has finite overlap, we obtain
\begin{align*}
\sum_{I \subseteq J}\frac{K\left(\frac{|I|}{|J|}\right)}{K(|I|)}d_{I,K}(f)
&\lesssim  \sum_{I \subseteq J} K\left(\frac{|I|}{|J|}\right) \int_{\widetilde{Q_I^{up}}} |f'(z)|^2 dA(z)\\
& \simeq \sum_{I \subseteq J} \int_{\widetilde{Q_I^{up}}} |f'(z)|^2 K\left(\frac{1-|z|}{|J|}\right) dA(z)\\
& \lesssim \|f\|_{Q_{K,*}}^2,
\end{align*}
whence
\begin{align*}
\sum_{I \subseteq J}\frac{K\left(\frac{|I|}{|J|}\right)}{K(|I|)} \frac{d_{I,K}(f)}{\|f\|_{Q_{K,*}}^2} \lesssim 1.
\end{align*}

Note that
\begin{align*}
\int_{\mathbb{D}} \big|f'(z)\big|^q W(z)\,dA(z)
&= \sum_{I\in \mathcal{D}} \int_{Q_I^{\text{up}}} \big|f'(z)\big|^q W(z)\,dA(z)\\
&\le \sum_{I\in \mathcal{D}} \left( \sup_{z\in Q_I^{\text{up}}} \big|f'(z)\big|^q \right) \cdot \int_{Q_I^{\text{up}}} W(z)\,dA(z)\\
& = \sum_{I\in \mathcal{D}} \left[ \sup_{z\in Q_I^{\text{up}}} \big|f'(z)\big|^2 (|I|^2 K(|I|)) \right]^{q/2} \frac{\int_{Q_I^{\text{up}}} W(z)\,dA(z)}{(|I|^2 K(|I|))^{q/2}}\\
& = \sum_{I\in \mathcal{D}} [d_{I,K}(f)]^{q/2} h_{W,q}^K(I),
\end{align*}
Thus, by the definition of the $Q_K$-capacity gauge, we have
\begin{align*}
\int_{\mathbb{D}} \big|f'(z)\big|^q W(z)\,dA(z)
\lesssim \|h_{W,q}^K\|_{B_K^q(\HD)} \|f\|_{Q_{K,*}}^q.
\end{align*}
This finishes the sufficiency part.

\subsection{Proof of Theorem \ref{t1}: Necessity for boundedness}

First, let
\begin{align*}
\HT_{K}:= \left\{ G \text{ measurable} : \ \|G\|_{\HT_{K}}^2 = \sup_{J\subseteq \mathbb{T}} \int_{Q_J} |G(z)|^2 K\left( \frac{1-|z|}{|J|} \right) dA(z) < +\infty \right\},
\end{align*}
so that
$$
f \in Q_K\iff f' \in \HT_{K}\ \&\ \|f\|_{Q_{K,*}} \simeq \|f'\|_{\HT_{K}},
$$
see \cite{EWX06, WZ17}.
Also, let
\[
\begin{cases}
b\gg 1;\\
B_b G(z) = \int_{\mathbb{D}} \frac{G(w)}{\big(1 - z \overline{w}\big)^{b+2}} dA_b(w);\\
dA_b(w) = C_b \big(1 - |w|^2\big)^b dA(w);\\
\text{$C_b>0$ = a normalizing constant}.
\end{cases}
\]

We first record six lemmas.

\vspace{0.1cm}

\noindent $\bullet$ The first is \cite[Theorem~3.9]{WZ17}, stated below.
\begin{lem}\label{integralestK}
Let $K$ satisfy conditions (a)--(d). Let $J \subseteq \T$ be an arc. If $s + \alpha \geq 1 + \sigma$, $s \geq \sigma$, $\alpha > 0$, and $K(t) = K(1)$ for $t \geq 1$, then there exist some $\beta \in (0,1)$ and a positive constant $C$ (independent of $J$ and $w$) such that
\[
\int_{\mathbb{D}} \frac{K\left(\dfrac{1-|z|}{|J|}\right) \big(1 - |w|^2\big)^{s-1}}{(1 - |z|)^{1-\alpha+\beta} \big|1 - \overline{w}z\big|^{s+\alpha}} dA(z)
\leq C \frac{K\left(\dfrac{1-|w|}{|J|}\right)}{(1 - |w|)^\beta}
\  \  \forall\  \
w \in \mathbb{D},
\]
where $\sigma > 0 $ is the number in condition (c).
\end{lem}
In fact, the proof of \cite[Theorem~3.9]{WZ17} shows that $\beta$ may be any number in $(0, \min\{\alpha, 1\})$ such that $s - \sigma + \beta + c > 1$, where $0 < c < \sigma$ is sufficiently small.

\vspace{0.1cm}

\noindent $\bullet$ The second lemma concerns the boundedness of an integral operator.
\begin{lem}\label{SchurK}
Let
$$
\begin{cases}
\text{$J \subseteq \T$ be an arc};\\
b\gg 1;\\
(z,w)\in\mathbb D^2;\\
H_J(z,w) = \frac{c_b (1-|w|^2)^b}{(1-z\overline{w})^{b+2}} \frac{\left[K\left( \frac{1-|z|}{|J|}\right)\right]^{1/2}}{\left[K\left( \frac{1-|w|}{|J|}\right)\right]^{1/2}};\\
K_J(z,w) = |H_J(z,w)|;\\
Tf(z) = \int_\D K_J(z,w) f(w) dA(w).
\end{cases}
$$
Then $T$ is bounded on $L^2(\D)$.
\end{lem}
\begin{proof}
We use the Schur test (\cite[Corollary 3.7]{Z07}) to prove the lemma.

For a fixed number $0 < \eta < 1$, let
$$h_J(z) = \frac{K^{1/2}\left( \frac{1-|z|}{|J|}\right)}{(1-|z|)^\eta}.
$$
We first have
\begin{align*}
\int_\D K_J(z,w) h_J(w) dA(w)
& = \int_\D \frac{c_b (1-|w|^2)^b}{|1-z\overline{w}|^{b+2}} \frac{\left[K\left( \frac{1-|z|}{|J|}\right)\right]^{1/2}}{\left[K\left( \frac{1-|w|}{|J|}\right)\right]^{1/2}} \frac{K^{1/2}\left( \frac{1-|w|}{|J|}\right)}{(1-|w|)^\eta} dA(w)\\
& = K^{1/2}\left( \frac{1-|z|}{|J|}\right)\int_\D \frac{c_b (1-|w|^2)^b}{|1-z\overline{w}|^{b+2}} \frac{1}{(1-|w|)^\eta} dA(w)\\
& \lesssim \frac{K^{1/2}\left( \frac{1-|z|}{|J|}\right)}{(1-|z|)^\eta} \quad \text{ by \cite[Lemma 3.10]{Z07}}\\
& = h_J(z),
\end{align*}
and also
\begin{align*}
\int_\D K_J(z,w) h_J(z) dA(z)& = \int_\D \frac{c_b (1-|w|^2)^b}{|1-z\overline{w}|^{b+2}} \frac{\left[K\left( \frac{1-|z|}{|J|}\right)\right]^{1/2}}{\left[K\left( \frac{1-|w|}{|J|}\right)\right]^{1/2}} \frac{K^{1/2}\left( \frac{1-|z|}{|J|}\right)}{(1-|z|)^\eta} dA(z)\\
& = \frac{c_b (1-|w|^2)^b}{K^{1/2}\left( \frac{1-|w|}{|J|}\right)}\int_\D \frac{K\left( \frac{1-|z|}{|J|}\right)}{|1-z\overline{w}|^{b+2}} \frac{1}{(1-|z|)^\eta} dA(z)\\
& \lesssim \frac{c_b (1-|w|^2)^b}{K^{1/2}\left( \frac{1-|w|}{|J|}\right)} \frac{K\left(\dfrac{1-|w|}{|J|}\right)}{(1 - |w|)^{\beta+s-1}},
\end{align*}
where in the last inequality we used Lemma \ref{integralestK} with
$$
\begin{cases} s + \alpha = b+2;\\
1 - \alpha + \beta = \eta.
\end{cases}
$$
 We may take $b > 1$ sufficiently large so that
$$s + \alpha = b+2 \geq 1 + \sigma,\  \  s \geq \sigma,\  \  \alpha > 0$$
and
$$\beta = \eta + \alpha - 1 \in (0, \min\{\alpha, 1\})\ \text{such that}\ s - \sigma + \beta + c > 1,
$$
where $0 < c < \sigma$ is sufficiently small. Then by
$$\beta + s -1 - b = \eta,$$
we obtain
$$\int_\D K_J(z,w) h_J(z) dA(z) \lesssim \frac{K^{1/2}\left( \frac{1-|w|}{|J|}\right)}{(1-|w|)^\eta} = h_J(w).$$
Thus $T$ is bounded on $L^2(\D)$.
\end{proof}

\noindent $\bullet$ The third lemma is the following result.

\begin{lem}\label{bddnessbergproj}
For $b \gg 1$ there holds
\begin{align*}
\|B_b G\|_{\HT_{K}} \lesssim \|G\|_{\HT_{K}}.
\end{align*}
\end{lem}

\begin{proof}
Without loss of generality, we assume $\|G\|_{\HT_{K}} = 1$. It suffices to show that for any $J \subseteq \T$,
\begin{align*}
\int_{Q_J} |B_bG(z)|^2 K\left( \frac{1-|z|}{|J|} \right) dA(z) \lesssim 1.
\end{align*}

Let $J^{(n)} \subseteq \mathbb{T}$ be the arc with the same center as $J$ and with length
\[
|J^{(n)}| \simeq \min\{2^n |J|, 1\}.
\]
If $J^{(n)} = \mathbb{T}$, then $J^{(m)}$ equals $\mathbb{T}$ for all larger $m$. Write
\[
G_0 = G \one_{Q_{J^{(2)}}}, \qquad G_n = G \one_{Q_{J^{(n+1)}} \setminus Q_{J^{(n)}}}, \quad n \geq 2.
\]
Then
\begin{equation*}
B_b G(z) = B_b G_0(z) + \sum_{n\geq 2} B_b G_n(z), \qquad z \in \mathbb{D}.
\end{equation*}
Thus, it suffices to show that
\begin{align}\label{Bbg0bdd}
\int_{Q_J} |B_bG_0(z)|^2 K\left( \frac{1-|z|}{|J|} \right) dA(z) \lesssim 1,
\end{align}
and
\begin{align}\label{Bbsumgnbdd}
\int_{Q_J} \left|\sum_{n=2}^\infty B_bG_n(z)\right|^2 K\left( \frac{1-|z|}{|J|} \right) dA(z) \lesssim 1.
\end{align}

We first prove (\ref{Bbg0bdd}). Let
\begin{align*}
H_J(z,w) = \frac{c_b (1-|w|^2)^b}{(1-z\overline{w})^{b+2}} \frac{\left[K\left( \frac{1-|z|}{|J|}\right)\right]^{1/2}}{\left[K\left( \frac{1-|w|}{|J|}\right)\right]^{1/2}}, \quad K_J(z,w) = |H_J(z,w)|.
\end{align*}
We have
\begin{align*}
&\int_{Q_J} |B_bG_0(z)|^2 K\left( \frac{1-|z|}{|J|} \right) dA(z)\\
& \leq \int_\D |B_bG_0(z)|^2 K\left( \frac{1-|z|}{|J|} \right) dA(z)\\
& = \int_\D \left|\int_\D G_0(w) \frac{c_b (1-|w|^2)^b}{(1-z\overline{w})^{b+2}} K^{1/2}\left( \frac{1-|z|}{|J|} \right) dA(w) \right|^2 dA(z)\\
& = \int_\D \left|\int_\D G_0(w) K^{1/2}\left( \frac{1-|w|}{|J|} \right) H_J(z,w) dA(w) \right|^2 dA(z)\\
& \leq \int_\D \left|\int_\D |G_0(w)| K^{1/2}\left( \frac{1-|w|}{|J|} \right) K_J(z,w) dA(w) \right|^2 dA(z)\\
& \lesssim \int_\D |G_0(w)|^2 K\left( \frac{1-|w|}{|J|} \right)  dA(w) \quad (\text{ by Lemma \ref{SchurK}}\ )\\
& = \int_{Q_{J^{(2)}}} |G(w)|^2 K\left( \frac{1-|w|}{|J|} \right)  dA(w)\\
& \lesssim \int_{Q_{J^{(2)}}} |G(w)|^2 K\left( \frac{1-|w|}{|J^{(2)}|} \right)  dA(w) \quad (\text{ by the fact that}\  K \text{ is doubling}\ )\\
& \lesssim \|G\|_{\HT_K}^2 = 1.
\end{align*}
Now we prove (\ref{Bbsumgnbdd}). By Minkowski inequality, we have
\begin{align}\label{minkowskiestsum}
&\int_{Q_J} \left|\sum_{n=2}^\infty B_bG_n(z)\right|^2 K\left( \frac{1-|z|}{|J|} \right) dA(z) \notag\\
& \lesssim \left[ \sum_{n=2}^\infty \left(\int_{Q_J} |B_b G_n(z)|^2 K\left( \frac{1-|z|}{|J|} \right) dA(z) \right)^{1/2}\right]^2.
\end{align}
Since
$$
\begin{cases} G_n = G \one_{Q_{J^{(n+1)}} \setminus Q_{J^{(n)}}};\\
|1 - z\overline{w}| \simeq (1 - |z|) + (1 - |w|) + \left| \frac{z}{|z|} - \frac{w}{|w|} \right|
\quad \forall\quad z,w \in \mathbb{D}\setminus\{0\},
\end{cases}
$$
we have
\[
|1 - z\overline{w}| \gtrsim \min\{2^n |J|, 1\}\  \text{when $z \in Q_J$ and $w \in Q_{J^{(n+1)}} \setminus Q_{J^{(n)}}$},
\]
and hence
\begin{align*}
&|B_b G_n(z)| \\
&\lesssim \int_{Q_{J^{(n+1)}}\setminus Q_{J^{(n)}}} \frac{|G(w)|}{|1-z\overline{w}|^{b+2}} (1-|w|^2)^b dA(w)\\
&\lesssim \frac{1}{(\min\{2^n |J|, 1\})^{b+2}} \int_{Q_{J^{(n+1)}}} |G(w)| (1-|w|^2)^b dA(w)\\
&\leq \frac{1}{(\min\{2^n |J|, 1\})^{b+2}}
\left(\int_{Q_{J^{(n+1)}}} |G(w)|^2
K\left( \frac{1-|w|}{|J^{(n+1)}|} \right) dA(w)\right)^{1/2}
\notag\\
&\qquad {}\times
\left(\int_{Q_{J^{(n+1)}}}
\frac{(1-|w|^2)^{2b}}
{K\left( \frac{1-|w|}{|J^{(n+1)}|} \right)}
dA(w)\right)^{1/2}\\
&\lesssim \frac{1}{\min\{2^n |J|, 1\}} \left(\int_{Q_{J^{(n+1)}}} |G(w)|^2 K\left( \frac{1-|w|}{|J^{(n+1)}|} \right) dA(w)\right)^{1/2}\\
&\lesssim \frac{\|G\|_{\HT_K}}{\min\{2^n |J|, 1\}} = \frac{1}{\min\{2^n |J|, 1\}},
\end{align*}
where in the second to the last estimate we used Lemma \ref{estimateintegralnpo} below. Furthermore, by (\ref{minkowskiestsum}), we have
\begin{align*}
&\int_{Q_J} \left|\sum_{n=2}^\infty B_bG_n(z)\right|^2 K\left( \frac{1-|z|}{|J|} \right) dA(z) \\
& \lesssim \left[ \sum_{n=2}^\infty \left(\int_{Q_J} |B_b G_n(z)|^2 K\left( \frac{1-|z|}{|J|} \right) dA(z) \right)^{1/2}\right]^2\\
& \lesssim \left[ \sum_{n=2}^\infty \left(\int_{Q_J}  \frac{1}{(\min\{2^n |J|, 1\})^2} K\left( \frac{1-|z|}{|J|} \right) dA(z) \right)^{1/2}\right]^2\\
& \lesssim \left[ \sum_{n=2}^\infty \frac{1}{\min\{2^n |J|, 1\}}\left(\int_{Q_J}  K(1) dA(z) \right)^{1/2}\right]^2\\
& \lesssim \left[ \sum_{n=2}^\infty \frac{|J|}{\min\{2^n |J|, 1\}}\right]^2\\
& \lesssim 1,
\end{align*}
Indeed, the terms for which $2^n|J|\leq 1$ form a geometric series, while
only $O(1)$ nonzero terms occur once $2^n|J|>1$.
This proves (\ref{Bbsumgnbdd}) and finishes the proof.
\end{proof}

\noindent $\bullet$ The fourth lemma is the following result.
\begin{lem}\label{estimateintegralnpo}
Let $J \subseteq \T$ be an arc. If $J^{(n)} \subseteq \mathbb{T}$ is the arc with the same center as $J$ and with length
\[
|J^{(n)}| \simeq \min\{2^n |J|, 1\},
\]
then
\begin{align*}
\int_{Q_{J^{(n+1)}}} \frac{(1-|w|^2)^{2b}}{K\left( \frac{1-|w|}{|J^{(n+1)}|} \right)} dA(w) \lesssim |J^{(n+1)}|^{2b+2} \simeq (\min\{2^n |J|, 1\})^{2b+2}~~\ \forall\ b\gg 1.
\end{align*}
\end{lem}
\begin{proof}
We have
\begin{align*}
\int_{Q_{J^{(n+1)}}} \frac{(1-|w|^2)^{2b}}{K\left( \frac{1-|w|}{|J^{(n+1)}|} \right)} dA(w)
& \simeq |J^{(n+1)}| \int_0^{|J^{(n+1)}|} \frac{t^{2b}}{K\left( \frac{t}{|J^{(n+1)}|}\right)}dt\\
& = |J^{(n+1)}| \int_0^1 \frac{|J^{(n+1)}|^{2b}t^{2b}}{K(t)} |J^{(n+1)}| dt\\
& = |J^{(n+1)}|^{2b+2} \int_0^1 \frac{t^{2b}}{K(t)}dt.
\end{align*}
Now we show
$$
\int_0^1 \frac{t^{2b}}{K(t)}dt < \infty~~\  \forall\  b \gg 1.
$$

Recall that the function
\[
(0,\infty)\ni s\mapsto \varphi_K(s) = \sup_{0 < t \le 1} \frac{K(s t)}{K(t)}
\]
satisfies
\[
\int_{1}^{+\infty} \frac{\varphi_K(s)}{s^{1+\sigma}} ds < +\infty \quad \text{for some } \sigma>0.
\]
So for any \(0 < t \le 1\), we have
\[
\varphi_K\!\left(\frac{1}{t}\right) \ge \frac{K(1)}{K(t)}\  \&\
\frac{1}{K(t)} \le \frac{1}{K(1)} \cdot \varphi_K\!\left(\frac{1}{t}\right).
\]
It follows that
\begin{align}\label{estimateintegraltabkt}
\int_0^1 \frac{t^{2b}}{K(t)}dt&\leq \int_0^1 t^{2b} \frac{1}{K(1)} \varphi_K\left(\frac{1}{t}\right) dt\\
& = \frac{1}{K(1)} \int_1^\infty \frac{1}{s^{2b}} \varphi_K(s) \frac{ds}{s^2} \notag\\
& = \frac{1}{K(1)} \int_1^\infty \frac{\varphi_K(s)}{s^{2b+2}} ds \notag\\
& \leq \frac{1}{K(1)} \int_1^\infty \frac{\varphi_K(s)}{s^{1+\sigma}} ds < \infty \notag,
\end{align}
where in the last estimate we used
$$b\gg 1\  \&\  2b+2 \geq 1 + \sigma.
$$
This estimate is consistent with what we have used in the proof of Lemma \ref{SchurK}.
\end{proof}

\noindent $\bullet$ For the fifth lemma, let $\delta > 0$ be sufficiently small. For each $I \in \mathcal{D}$, choose a finite partition
\begin{equation*}
Q_I^{\mathrm{up}} = \bigcup_{\nu=1}^{N_0} Q_{I,\nu}
\end{equation*}
where $N_0 = N_0(\delta)$ is independent of $I$, and each $Q_{I,\nu}$ is a ``truncated'' Whitney-type tent with Euclidean diameter at most $\delta|I|$.
We choose $\delta$ sufficiently small so that, on each product $Q_{I,\nu} \times Q_{I,\nu}$, the kernel $\frac{1}{(1-z\overline{w})^{b+2}}$ has essentially constant argument.

\begin{lem}\label{whitneyest}
Let
\begin{equation}\label{defbinv}
\begin{cases}
A_K(Q_{I,\nu}) = \int_{Q_{I,\nu}} K(1-|z|) dA(z);\\
b_{I,\nu}= \big(A_K(Q_{I,\nu})\big)^{-\frac12}\one_{Q_{I,\nu}}.
\end{cases}
\end{equation}
Then for $b\gg 1$ there holds
$$
\big|B_b b_{I,\nu}(z)\big| \gtrsim \frac{1}{|I|\,K(|I|)^{1/2}}, \qquad \forall z\in Q_{I,\nu}.
$$
\end{lem}
\begin{proof}
By the definition of $b_{I,\nu}$,
\begin{align*}
\big|B_b b_{I,\nu}(z)\big|& =
\left| \int_{Q_{I,\nu}}\frac{1}{A_K(Q_{I,\nu})^{1/2}} \frac{c_b(1-|w|^2)^b}{(1-z\overline{w})^{b+2}}dA(w)\right|\\
& \gtrsim \int_{Q_{I,\nu}}\frac{1}{A_K(Q_{I,\nu})^{1/2}} \frac{|I|^b}{(|I|)^{b+2}}dA(w)\\
& \simeq \frac{1}{A_K(Q_{I,\nu})^{\frac12}},
\end{align*}
Moreover,
\begin{align*}
A_K(Q_{I,\nu})& = \int_{Q_{I,\nu}} K(1-|z|) dA(z)\simeq K(|I|) \int_{Q_{I,\nu}} dA(z) \simeq K(|I|) |I|^2,
\end{align*}
Combining these estimates, we get
$$\big|B_b b_{I,\nu}(z)\big| \gtrsim \frac{1}{|I|\,K(|I|)^{1/2}}\quad\forall\quad z\in Q_{I,\nu},
$$
as desired.
\end{proof}

\noindent $\bullet$ The sixth lemma is the following result.

\begin{lem}\label{packinglem}
Let $\{d_I\}_{I \in \HD}$ be a nonnegative sequence satisfying
\begin{align*}
\sup_{J \in \HD} \sum_{I \subseteq J} \frac{K\left( \frac{|I|}{|J|}\right)}{K(|I|)}d_I \leq 1.
\end{align*}
If
$E \subseteq \mathbb{T}$ is any arc with
\[
\mathcal{I}(E) = \bigl\{ I \in \mathcal{D} : Q_I^{\mathrm{up}} \cap Q_E \neq \emptyset \bigr\},
\]
then
\begin{align*}
\sum_{I \in \mathcal{I}(E)} \frac{K\left( \frac{|I|}{|E|}\right)}{K(|I|)}d_I \lesssim 1.
\end{align*}
\end{lem}
\begin{proof}
Since
$$
Q_I^{\mathrm{up}} \cap Q_E \neq \emptyset\Rightarrow |I| \lesssim |E|,
$$
we consider two cases.

\smallskip
\noindent\emph{Case 1: $I\subseteq E$.}
Let $M(E)$ denote the collection of maximal dyadic subintervals of $E$. The intervals in $M(E)$ are pairwise disjoint,
and every dyadic interval $I \subseteq E$ is contained in a unique $J \in M(E)$. Moreover,
\begin{align*}
\sum_{I \subseteq J} \frac{K\left( \frac{|I|}{|J|}\right)}{K(|I|)}d_I \leq 1.
\end{align*}
The maximal dyadic subintervals of $E$ form at most two Whitney chains and one central interval, as in \cite{HZ26}.
Hence
\begin{align*}
\sum_{I \subseteq E, I \in \mathcal{I}(E)} \frac{K\left( \frac{|I|}{|E|}\right)}{K(|I|)}d_I
&= \sum_{J \in M(E)} \sum_{I \subseteq J}\frac{K\left( \frac{|I|}{|E|}\right)}{K(|I|)}d_I\\
& = \sum_{J \in M(E)} \sum_{I \subseteq J}\frac{K\left( \frac{|J|}{|E|} \frac{|I|}{|J|}\right)}{K(|I|)}d_I\\
& \leq \sum_{J \in M(E)} \sum_{I \subseteq J} \frac{\varphi_K\left( \frac{|J|}{|E|}\right) K\left( \frac{|I|}{|J|}\right)}{K(|I|)}d_I \quad (\text{via}\  \varphi_K(s) = \sup_{0<t\le 1} \frac{K(st)}{K(t)}\ )\\
& \leq \sum_{J \in M(E)} \varphi_K\left( \frac{|J|}{|E|}\right) \quad (\text{ by assumption}\ )\\
& \lesssim \sum_{j=1}^\infty \varphi_K (2^{-j})\\
& \lesssim \int_{0}^{1} \frac{\varphi_K(s)}{s} ds < \infty.
\end{align*}

\smallskip
\noindent\emph{Case 2: $I\not\subseteq E$.}
At each dyadic scale, there are only $O(1)$ such intervals meeting one of the endpoints of $E$. By assumption, $d_I \leq K(|I|)$. Hence
\begin{align*}
\sum_{I \not\subseteq E, I \in \mathcal{I}(E)} \frac{K\left( \frac{|I|}{|E|}\right)}{K(|I|)}d_I
&\leq \sum_{I \not\subseteq E, I \in \mathcal{I}(E)} K\left( \frac{|I|}{|E|}\right)\\
& \lesssim \sum_{I \not\subseteq E, I \in \mathcal{I}(E)} \varphi_K \left( \frac{|I|}{|E|}\right)\\
& \lesssim \sum_{j=1}^\infty \varphi_K (2^{-j}) < \infty.
\end{align*}
\end{proof}

Finally, we complete the argument for necessity of the boundedness.

Assume
\begin{align*}
\int_\D |f'(z)|^q W(z) dA(z) \leq C_{best} \|f\|_{Q_{K,*}}^q\quad\forall\quad f \in Q_K
\end{align*}
where $C_{\mathrm{best}}$ is the best constant.
It suffices to show that for any non-negative sequence $\{d_I\}_{I\in\mathcal{D}}$ with
\begin{align*}
\sup_{J \in \HD} \sum_{I \subseteq J} \frac{K\left( \frac{|I|}{|J|}\right)}{K(|I|)}d_I \leq 1,
\end{align*}
one has
$$
\sum_{I \in \HD} d_I^{q/2} h_{W,q}^K(I) \lesssim C_{best},
\quad\text{where}\quad h_{W,q}^K(I) = \frac{\int_{Q_I^{\text{up}}} W(z)\,dA(z)}{(|I|^2 K(|I|))^{q/2}}.
$$

Let
$$\varepsilon = \{\varepsilon_{I,\nu}: I\in\mathcal{D}, 1\le\nu\le N_0\}$$
be independent Rademacher random variables. Suppose that
$$
\begin{cases}
G_\varepsilon(w) = \sum_{I\in\mathcal{D}} \sum_{\nu=1}^{N_0} \varepsilon_{I,\nu} d_I^{1/2} b_{I,\nu}(w)\quad \forall\quad w\in\mathbb{D};\\
\text{$b_{I,\nu} = \big({A_K(Q_{I,\nu})\big)^{-\frac12}} \one_{Q_{I,\nu}}$ is given by (\ref{defbinv}).}
\end{cases}
$$
By the disjointness of $\{Q_{I,\nu}\}$, we have
\[
|G_\varepsilon(w)|^2 = \sum_{I\in\mathcal{D}} \sum_{\nu=1}^{N_0} d_I |b_{I,\nu}(w)|^2\quad\forall\quad w\in\mathbb{D}.
\]
We claim that
$$
G_\varepsilon \in \HT_K
\quad\text{for every choice of signs $\varepsilon$}.
$$

For any arc $E \subseteq \T$, recall that
$$\mathcal{I}(E) = \bigl\{ I \in \mathcal{D} : Q_I^{\mathrm{up}} \cap Q_E \neq \emptyset \bigr\}.
$$
So we have
\begin{align*}
&\int_{Q_E} |G_\varepsilon(w)|^2 K\left( \frac{1-|w|}{|E|}\right)dA(w)\\
& \leq \sum_{I\in \mathcal{I}(E)} \sum_{\nu=1}^{N_0} d_I \int_{Q_{I,\nu}} |b_{I,\nu}(w)|^2 K\left( \frac{1-|w|}{|E|}\right)dA(w)\\
& = \sum_{I\in \mathcal{I}(E)} \sum_{\nu=1}^{N_0} \frac{d_I}{A_K(Q_{I,\nu})} \int_{Q_{I,\nu}} K\left( \frac{1-|w|}{|E|}\right)dA(w)\\
& \simeq \sum_{I\in \mathcal{I}(E)} \sum_{\nu=1}^{N_0} \frac{d_I}{|I|^2 K(|I|)} K\left( \frac{|I|}{|E|}\right) |I|^2\\
& \lesssim \sum_{I\in \mathcal{I}(E)} \frac{K\left( \frac{|I|}{|E|}\right)}{K(|I|)} d_I \\
& \lesssim 1 \quad (\text{ by Lemma \ref{packinglem}}\ ).
\end{align*}
Thus $\|G_\varepsilon\|_{\HT_K} \lesssim 1$. This, along with Lemma \ref{bddnessbergproj}, implies
$$\|B_b G_\varepsilon\|_{\HT_K} \lesssim \|G_\varepsilon\|_{\HT_K} \lesssim 1.
$$

Next let
\begin{align*}
F_\varepsilon(z): = \int_0^z B_b G_\varepsilon(\xi) d\xi.
\end{align*}
Then
$$
\begin{cases} F_\varepsilon \in Q_K;\\
F_\varepsilon(0) = 0;\\
\|F_\varepsilon\|_{Q_{K,*}} = \|B_b G_\varepsilon\|_{\HT_K} \lesssim 1.
\end{cases}
$$
Thus, by assumption, we have
\begin{align*}
\int_\D |B_b G_\varepsilon(z)|^q W(z) dA(z) = \int_\D |F_\varepsilon'(z)|^q W(z) dA(z) \lesssim C_{best}.
\end{align*}
Now Khinchin's inequality yields
\begin{align*}
\mathbb{E}_\varepsilon |B_b G_\varepsilon(z)|^q \simeq\left(\sum_{I\in\mathcal{D}} \sum_{\nu=1}^{N_0} d_I |B_b b_{I,\nu}(z)|^2\right)^{q/2}.
\end{align*}
Thus
\begin{align*}
C_{best}&= \mathbb{E}_\varepsilon (C_{best})\\
& \gtrsim \mathbb{E}_\varepsilon \int_\D |B_b G_\varepsilon(z)|^q W(z) dA(z)\\
& \gtrsim \int_\D \left(\sum_{I\in\mathcal{D}} \sum_{\nu=1}^{N_0} d_I |B_b b_{I,\nu}(z)|^2\right)^{q/2} W(z) dA(z)\\
& = \sum_{J\in \HD} \sum_{\widetilde{\nu}=1}^{N_0} \int_{Q_{J,\widetilde{\nu}}} \left(\sum_{I\in\mathcal{D}} \sum_{\nu=1}^{N_0} d_I |B_b b_{I,\nu}(z)|^2\right)^{q/2} W(z) dA(z)\\
& \geq \sum_{J\in \HD} \sum_{\widetilde{\nu}=1}^{N_0} \int_{Q_{J,\widetilde{\nu}}} d_J^{q/2} |B_b b_{J,\widetilde{\nu}}(z)|^q W(z) dA(z) \quad \text{(retaining only the summand $I=J$)}\\
& \gtrsim \sum_{J\in \HD} \sum_{\widetilde{\nu}=1}^{N_0} \frac{d_J^{q/2}}{(|J|^2K(|J|))^{q/2}}\int_{Q_{J,\widetilde{\nu}}}  W(z) dA(z) \  \text{(by Lemma \ref{whitneyest})}\\
& = \sum_{J\in \HD} \frac{d_J^{q/2}}{(|J|^2K(|J|))^{q/2}}\int_{Q_{J}^{up}}  W(z) dA(z)\\
& = \sum_{J\in \HD} d_J^{q/2} h_{W,q}^K(J).
\end{align*}
This completes the proof.

\section{Proof of Theorem~\ref{t1}: Compactness}\label{s3}
Recall that, for a complete metric space $X$, a bounded linear operator $T:Q_K\rightarrow X$ is compact provided that $\|Tf_n\|_X\rightarrow0$ for every bounded sequence $\{f_n\}\subseteq Q_K$ converging to $0$ uniformly on compact subsets of $\D$. We divide the proof into two parts.

\subsection{Proof of Theorem~\ref{t1}: Sufficiency for compactness}

Suppose that
$$
\frac{d}{dz} : Q_K \to L^q(W\,dA)
\  \text{
is bounded}
$$
and
\begin{equation}\label{31}
\lim_{\rho\to 1^-} \sup_{\mathcal{D}} \big\|h_{W_\rho, q}^K\big\|_{\mathcal{B}_K^q(\mathcal{D})} = 0.
\end{equation}
Let $\{f_n\} \subseteq Q_K$ be a bounded sequence such that $f_n \rightarrow 0$ uniformly on compact subsets of $\D$. Then $f_n'$ also converges to $0$ uniformly on compact subsets of $\D$.

Given $\varepsilon > 0$, choose $0<\rho <1$ such that $$\sup_{\mathcal{D}} \big\|h_{W_\rho, q}^K\big\|_{\mathcal{B}_K^q(\mathcal{D})} < \varepsilon.
$$
Then, by Theorem~\ref{t1}(i), we obtain
\begin{align*}
\int_{|z|>\rho} |f_n'(z)|^q W(z) dA(z) = \int_{\D} |f_n'(z)|^q W_\rho(z) dA(z) \lesssim \varepsilon \|f_n\|_{Q_{K,*}}^q \lesssim \varepsilon.
\end{align*}
Since $f_n' \rightarrow 0$ uniformly on compact subsets of $\D$, there exists $N$ such that, for $n > N$,
\begin{align*}
\int_{|z|\leq \rho} |f_n'(z)|^q W(z) dA(z) \leq \varepsilon \int_{|z|\leq \rho} W(z) dA(z) \lesssim \varepsilon,
\end{align*}
it follows that
\begin{align*}
\lim_{n \rightarrow \infty}\int_{\D} |f_n'(z)|^q W(z) dA(z) = 0,
\end{align*}
and hence
$$
\frac{d}{dz} : Q_K \to L^q(W\,dA)\  \text{is compact.}
$$

\subsection{Proof of Theorem~\ref{t1}: Necessity for compactness}

Suppose that $$\frac{d}{dz} : Q_K \to L^q(W\,dA)\  \text{is compact.}
$$ Then this operator is bounded. Consequently, we need to show that \eqref{31} holds. We prove it by contradiction. Suppose that \eqref{31} fails. Then there are $\rho_n \rightarrow 1$, dyadic systems $\HD_n$, and $c_0 > 0$ such that
$$\big\|h_{W_{\rho_n}, q}^K\big\|_{\mathcal{B}_K^q(\mathcal{D}_n)} \geq c_0.$$
Thus for each $n \geq 1$, there is a finitely supported nonnegative sequence $\{d_{I,n}\}_{I \in \HD_n}$ such that
$$\sup_{J \in \HD_n} \sum_{I \subseteq J} \frac{K\left(\frac{|I|}{|J|}\right)}{K(|I|)}d_{I,n}  \leq 1$$
and
\begin{align}\label{sumcondicontr}
\sum_{I \in \HD_n} d_{I,n}^{q/2} h_{W_{\rho_n}, q}^K(I) \geq \frac{c_0}{2}.
\end{align}
Since the support of $W_{\rho_n}$ is $\{|z| \geq \rho_n\}$, we may assume that
$$d_{I,n} = 0\ \text{whenever}\  Q_I^{\mathrm{up}} \cap \{ |z| > \rho_n \} = \emptyset.
$$

Let
$$\varepsilon = \{\varepsilon_{I,\nu}: I\in\mathcal{D}_n, 1\le\nu\le N_0\}$$
be independent Rademacher random variables. If
$$
\begin{cases}
G_{\varepsilon,n}(w) = \sum_{I\in\mathcal{D}_n} \sum_{\nu=1}^{N_0} \varepsilon_{I,\nu} d_{I,n}^{1/2} b_{I,\nu}(w)\ \forall\  w\in\mathbb{D};\\
F_{\varepsilon,n}(z)= \int_0^z B_b G_{\varepsilon,n}(\xi) d\xi\ \forall\ z\in\mathbb D;\\
b_{I,\nu} = {\big(A_K(Q_{I,\nu})\big)^{-\frac12}} \one_{Q_{I,\nu}}\  \text{ as given in (\ref{defbinv}),}
\end{cases}
$$
then, by the same argument as in the proof of Theorem~\ref{t1}(i), we have
\begin{align}\label{normestcont}
\|F_{\varepsilon,n}\|_{Q_{K,*}} = \|B_b G_{\varepsilon,n}\|_{\HT_K} \lesssim \|G_{\varepsilon,n}\|_{\HT_K} \lesssim 1,
\end{align}
and
\begin{align*}
\mathbb{E}_\varepsilon \int_\D |F_{\varepsilon,n}'(z)|^2 W_{\rho_n}(z) dA(z)& \gtrsim \sum_{I\in \HD_n} d_{I,n}^{q/2} h_{W_{\rho_n},q}^K(I) \geq \frac{c_0}{2} \quad \text{(by (\ref{sumcondicontr}))}.
\end{align*}
Thus for each $n$, there is a choice of signs $\varepsilon_n$ such that
\begin{align*}
\int_\D |F_{\varepsilon_n,n}'(z)|^2 W_{\rho_n}(z) dA(z) \gtrsim 1.
\end{align*}

Now let
$$
F_n(z) = F_{\varepsilon_n,n}(z).
$$ Then
\begin{equation}\label{contradineq}
\begin{cases} F_n(0) = 0;\\
 \|F_n\|_{Q_{K,*}} \lesssim 1;\\
\int_\D |F_{n}'(z)|^2 W(z) dA(z) \geq \int_\D |F_{n}'(z)|^2 W_{\rho_n}(z) dA(z) \gtrsim 1.
\end{cases}
\end{equation}
We claim that $F_n \rightarrow 0$ uniformly on compact subsets of $\D$. If the claim holds, then (\ref{contradineq}) contradicts the assumption that $$\frac{d}{dz} : Q_K \to L^q(W\,dA)\ \text{ is compact.}
$$
It remains to prove the claim.

Let $$G_n(z) = G_{\varepsilon_n,n}(z).
$$
Then $$F_n'(z) = B_b G_n(z).
$$
Moreover, if $$Q_I^{\mathrm{up}} \cap \{ |z| > \rho_n \} \neq \emptyset,$$
 then
\begin{equation*}
\exists\ \text{an absolute constant $C_0 \geq 1$ such that}\
Q_I^{\mathrm{up}} \subset \big\{ |z| > 1 - C_0(1 - \rho_n) \big\}.
\end{equation*}
Since we have assumed that
$$\text{$d_{I,n} = 0$ whenever $Q_I^{\mathrm{up}} \cap \{ |z| > \rho_n \} = \emptyset$,}
$$
we have
\begin{equation*}
\operatorname{supp} G_n \subset \big\{ |z| > 1 - C_0(1 - \rho_n) \big\}.
\end{equation*}
By (\ref{normestcont}), we have $\|G_n\|_{\HT_K} \lesssim 1$. In particular,
$$
\int_{\mathbb{D}} |G_n(w)|^2 K(1-|w|)\,dA(w) \lesssim 1.
$$

Let $V \subseteq \D$ be any compact set. For $z \in V$, the kernel $\frac{1}{(1-z\overline{w})^{b+2}}$ is uniformly bounded. Thus
\begin{align*}
|F_n'(z)|^2 &= |B_bG_n(z)|^2 \\
&= \left| \int_{\mathbb{D}} \frac{G_n(w)}{\big(1 - z \overline{w}\big)^{b+2}} dA_b(w)\right|^2\\
& \lesssim \left(\int_{\{|w| > 1 - C_0(1 - \rho_n)\}} |G_n(w)| (1-|w|^2)^b dA(w)\right)^2\\
& \leq  \int_{\mathbb{D}} |G_n(w)|^2 K(1-|w|)\,dA(w) \int_{\{|w| > 1 - C_0(1 - \rho_n)\}} \frac{(1-|w|^2)^{2b}}{K(1-|w|)}dA(w)\\
& \lesssim \int_0^{C_0(1 - \rho_n)} \frac{t^{2b}}{K(t)}dt.
\end{align*}
The last integral tends to $0$ as $n \rightarrow \infty$, since (\ref{estimateintegraltabkt}) gives
$$
\int_0^{1} \frac{t^{2b}}{K(t)}dt < \infty.
$$
Thus $F_n' \rightarrow 0$ uniformly on compact subsets of $\D$. Since $F_n(0) = 0$, we also have $F_n \rightarrow 0$ uniformly on compact subsets of $\D$.
The proof is complete.

\medskip

\section{Proof of Theorem \ref{20260726thm01}} \label{sec04}

It is clear that (ii) implies (i). Suppose that (i) holds. Since $Q_K$ is M\"obius invariant, $\sup_{a\in\D}\|\varphi_a\|_{Q_K}<\infty$, where $\varphi_a(z):=\frac{a-z}{1-\bar{a}z}$ is the standard disc automorphism associated to $a \in \D$.
Therefore,
$$
\sup_{a\in\D}
\|\varphi_a'\|_{L^\infty(W\,dA)}
\lesssim 1.
$$
Let $\{a_j\}_{j=1}^\infty$ be a countable dense subset of $\D$. Outside a $W\,dA$-measure zero set, we have $|\varphi_{a_j}'(z)|\lesssim 1$ for every $j\ge 1$. Since $|\varphi_a'(z)|=\frac{1-|a|^2}{|1-\overline{a}z|^2}$ depends continuously on $a$, it follows that, for $W\,dA$-almost every
$z\in\D$,
$$
\frac{1}{1-|z|^2}=\sup_{a\in\D}|\varphi_a'(z)|
=\sup_{j\ge 1}|\varphi_{a_j}'(z)| \lesssim 1.
$$
Consequently, the essential support of $W\,dA$ is contained in a disk
$\{z\in\D:|z|\le r\}$ for some $0<r<1$. Equivalently,
$$
W(z)=0
$$
for almost every $z\in\D$ with $|z|>r$. Thus (i) implies (iii).

It remains to prove that (iii) implies (ii). Suppose that
(iii) holds for some $0<r<1$. Since $Q_K$ is continuously contained in the Bloch space,
$$
|f'(z)| \lesssim \frac{\|f\|_{Q_K}}{1-|z|^2}.
$$
Hence
$$
\|f'\|_{L^\infty(W\,dA)} \le \sup_{|z|\le r}|f'(z)|
\lesssim_r \|f\|_{Q_K},
$$
which implies $d/dz$ is bounded.

\vspace{0.1cm}

To prove compactness, let $\{f_n\}_{n \ge 1}$ be a bounded sequence in $Q_K$. The
continuous inclusion $Q_K\subset\mathcal B$ shows that $\{f_n\}_{n \ge 1}$ is a normal
family. Thus, after passing to a subsequence, there exists an analytic
function $f$ on $\D$ such that $f_n\to f$, uniformly on compact subsets of $\D$. It follows that $f_n'\to f'$, uniformly on compact subsets of $\D$. In particular,
$$
\sup_{|z|\le r}|f_n'(z)-f'(z)|
\longrightarrow 0.
$$
Since $W=0$ almost everywhere on $\{z\in\D:|z|>r\}$, we obtain
$$
\|f_n'-f'\|_{L^\infty(W\,dA)}
\le
\sup_{|z|\le r}|f_n'(z)-f'(z)|
\longrightarrow 0.
$$
Therefore, $d/dz$ is compact, and the proof is
complete.

\section{Proof of Theorem \ref{20260728thm02}} \label{Sec05}

We shall use the following logarithmic version of Khinchine's inequality for
Steinhaus random variables. Recall that Steinhaus variables are the complex
counterparts of the Rademacher functions: they form an i.i.d. sequence
$\{\zeta_j\}$ of $\T$-valued random variables defined on a probability space
$(\Omega,P)$, each uniformly distributed on $\T$.

\begin{lem}\label{logkhinchine}
Let $\zeta_0,\ldots,\zeta_N$ be independent Steinhaus random variables. Then,
for every $c_0,\ldots,c_N\in\C$, not all zero,
$$
\mathbb{E}_{\zeta}
\log\left|\sum_{j=0}^N c_j\zeta_j\right|
\geq
\frac{1}{2}\log\left(\sum_{j=0}^N|c_j|^2\right)
-\frac{1}{2}\log 2.
$$
\end{lem}

\begin{proof}
By the sharp Khinchine inequality for Steinhaus variables
\cite[Theorem 1, Proposition 4]{Konig2014}, for $0<p<1$,
\begin{equation}\label{20260728eq37}
a_p\left(\sum_{j=0}^N|c_j|^2\right)^{1/2}
\leq
\left(
\mathbb{E}_{\zeta}
\left|\sum_{j=0}^Nc_j\zeta_j\right|^p
\right)^{1/p},
\end{equation}
where
$$
\lim_{p\to0^+}a_p=\frac{1}{\sqrt{2}}.
$$
Arguing as in \eqref{20260728eq38}--\eqref{20260728eq40} and then letting
$p\to0^+$ in \eqref{20260728eq37}, we deduce that
$$
\exp\left(
\mathbb{E}_{\zeta}
\log\left|\sum_{j=0}^Nc_j\zeta_j\right|
\right)
\geq
\frac{1}{\sqrt{2}}
\left(\sum_{j=0}^N|c_j|^2\right)^{1/2}.
$$
Taking logarithms on both sides completes the proof.
\end{proof}

\begin{proof}[Proof of Theorem~\ref{20260728thm02}]
We prove
$$
{\rm (iii)}\Longrightarrow{\rm (i)}
\Longrightarrow{\rm (ii)}
\Longrightarrow{\rm (iii)}.
$$

We first prove that {\rm (iii)} implies {\rm (i)}. Without loss of generality,
we may assume that $f$ is nonconstant, since otherwise the assertion is
trivial. Since $Q_K\subseteq\mathcal B$ under assumptions {\rm (a)--(d)}, we
have
$$
|f'(z)|
\lesssim
\frac{\|f\|_{Q_{K,*}}}{1-|z|^2},
\qquad z\in\D.
$$
Therefore, for every nonconstant $f\in Q_K$,
$$
\log|f'(z)|
\leq
\log\bigl(C_K\|f\|_{Q_{K,*}}\bigr)
+\log\frac{1}{1-|z|^2}
$$
for some constant $C_K>0$. By {\rm (iii)},
$$
\int_{\D}W(z)\log\frac{e}{1-|z|^2}\,dA(z)<\infty,
$$
and hence
\begin{align*}
\log M_{0,W}(f')
&=
\frac{1}{W(\D)}
\int_{\D}W(z)\log|f'(z)|\,dA(z)\\
&\leq
\log\bigl(C_K\|f\|_{Q_{K,*}}\bigr)
+
\frac{1}{W(\D)}
\int_{\D}W(z)\log\frac{1}{1-|z|^2}\,dA(z).
\end{align*}
It follows that
$$
M_{0,W}(f')
\leq
C_K
\exp\left(
\frac{1}{W(\D)}
\int_{\D}W(z)\log\frac{1}{1-|z|^2}\,dA(z)
\right)
\|f\|_{Q_{K,*}},
$$
which proves {\rm (i)}.

Since $Q_{K,0}\subseteq Q_K$, the implication
${\rm (i)}\Longrightarrow{\rm (ii)}$ is immediate.

\medskip

It remains to prove that {\rm (ii)} implies {\rm (iii)}. Suppose that
\begin{equation}\label{20260728eq01}
M_{0,W}(f')
\leq C\|f\|_{Q_{K,*}},
\qquad f\in Q_{K,0}.
\end{equation}
Fix a dyadic system $\HD$ on $\T$, and write
$$
\HD_n:=\{I\in\HD:|I|=2^{-n}\},
\qquad
\HD^{(N)}:=\bigcup_{n=0}^N\HD_n.
$$
For $I\in\HD$, define
$$
d_I^{(N)}
:=
|I|^{3/2}K(|I|)\one_{\HD^{(N)}}(I).
$$
We claim that
\begin{equation}\label{endpointpacking}
\sup_{J\in\HD}
\sum_{I\subseteq J}
\frac{K\left(\frac{|I|}{|J|}\right)}{K(|I|)}
d_I^{(N)}
\lesssim 1.
\end{equation}
Indeed, for every $J\in\HD$, using the assumption that $K$ is non-decreasing,
\begin{align*}
\sum_{I\subseteq J}
\frac{K\left(\frac{|I|}{|J|}\right)}{K(|I|)}
d_I^{(N)}
&\lesssim
\sum_{k=0}^{\infty}
\sum_{\substack{I\subseteq J\\ |I|=2^{-k}|J|}}
|I|^{3/2}K\left(\frac{|I|}{|J|}\right)\\
&=
\sum_{k=0}^{\infty}
\sum_{\substack{I\subseteq J\\ |I|=2^{-k}|J|}}
|I|^{3/2}K(2^{-k})\\
&\lesssim
|J|^{3/2}
\sum_{k=0}^{\infty}2^{-k/2}K(2^{-k})\\
&\lesssim
K(1)\sum_{k=0}^{\infty}2^{-k/2}.
\end{align*}

Let
$$
\left\{
\zeta_0,\zeta_{I,\nu}:
I\in\HD^{(N)},\ 1\leq\nu\leq N_0
\right\}
$$
be independent Steinhaus random variables, and define the truncated test
function
$$
G_{\zeta,N}(w)
:=
\sum_{I\in\HD^{(N)}}\sum_{\nu=1}^{N_0}
\zeta_{I,\nu}\bigl(d_I^{(N)}\bigr)^{1/2}
b_{I,\nu}(w).
$$
Since the sets $Q_{I,\nu}$ are pairwise disjoint,
$$
|G_{\zeta,N}(w)|^2
=
\sum_{I\in\HD^{(N)}}\sum_{\nu=1}^{N_0}
d_I^{(N)}|b_{I,\nu}(w)|^2.
$$
By \eqref{endpointpacking}, Lemma~\ref{packinglem}, and the same calculation
used in the necessity part of Theorem~\ref{t1},
$$
\|G_{\zeta,N}\|_{\HT_K}\lesssim1,
$$
uniformly in $N$ and $\zeta$. Lemma~\ref{bddnessbergproj} then gives
$$
\|B_bG_{\zeta,N}\|_{\HT_K}\lesssim1.
$$

Set\footnote{The auxiliary term $\zeta_0z$ is included both to justify the
later application of Fubini's theorem and to ensure the global lower bound
$\mathbb E_{\zeta}\log|F_{\zeta,N}'(z)|\geq0$; see
\eqref{20260728eq30}.}
$$
F_{\zeta,N}(z)
:=
\zeta_0z+\int_0^zB_bG_{\zeta,N}(\xi)\,d\xi.
$$
Then
$$
F_{\zeta,N}'(z)
=
\zeta_0+
\sum_{I\in\HD^{(N)}}\sum_{\nu=1}^{N_0}
\zeta_{I,\nu}\bigl(d_I^{(N)}\bigr)^{1/2}
B_bb_{I,\nu}(z),
$$
and
$$
\|F_{\zeta,N}\|_{Q_{K,*}}
\leq
\|z\|_{Q_{K,*}}
+\|B_bG_{\zeta,N}\|_{\HT_K}
\lesssim1.
$$

We next verify that the truncated test functions belong to $Q_{K,0}$. For
each fixed $N$, the function $G_{\zeta,N}$ is supported in a compact subset
of $\D$. Consequently, $B_bG_{\zeta,N}$, and hence $F_{\zeta,N}$, extends
analytically to a neighborhood of $\overline{\D}$. If $P_{m, N}$ denotes the
$m$-th Taylor polynomial of $F_{\zeta,N}$, then
$$
\sup_{z\in\D}|P_{m, N}'(z)-F_{\zeta,N}'(z)|\longrightarrow0.
$$
Since $K$ is non-decreasing,
\begin{align*}
\|P_{m, N}-F_{\zeta,N}\|_{Q_{K,*}}^2
&\lesssim
\sup_{a\in\D} \int_{\D} |P_{m, N}'(z)-F_{\zeta,N}'(z)|^2 K\bigl(1-|\sigma_a(z)|^2\bigr)\,dA(z)\\
&\lesssim K(1) \sup_{z\in\D}|P_{m, N}'(z)-F_{\zeta,N}'(z)|^2 \longrightarrow 0.
\end{align*}
Thus $F_{\zeta,N}$ belongs to the closure of the analytic polynomials in the
$Q_K$-norm, namely, $F_{\zeta,N}\in Q_{K,0}$.

The assumption \eqref{20260728eq01} now yields
\begin{equation}\label{20260728eq05}
\frac{1}{W(\D)}
\int_{\D}W(z)\log|F_{\zeta,N}'(z)|\,dA(z)
\lesssim1.
\end{equation}

For each fixed $z\in\D$, we apply Lemma~\ref{logkhinchine} to
$F_{\zeta,N}'(z)$, with coefficients
$$
c_0=1
\qquad\text{and}\qquad
c_{I,\nu}
:=
\bigl(d_I^{(N)}\bigr)^{1/2}B_bb_{I,\nu}(z).
$$
We obtain
\begin{equation}\label{20260728eq04}
\mathbb E_{\zeta}\log|F_{\zeta,N}'(z)|
\geq
\frac12
\log\left(
1+
\sum_{I\in\HD^{(N)}}\sum_{\nu=1}^{N_0}
d_I^{(N)}|B_bb_{I,\nu}(z)|^2
\right)
-\frac12\log2.
\end{equation}

We also have
\begin{equation}\label{20260728eq30}
\mathbb E_{\zeta}\log|F_{\zeta,N}'(z)|\geq0,
\qquad z\in\D.
\end{equation}
Indeed, fix $z\in\D$ and all the variables $\zeta_{I,\nu}$. Since the
Steinhaus variable $\zeta_0$ is uniformly distributed on $\T$, Jensen's
formula gives
$$
\begin{aligned}
\mathbb E_{\zeta_0}\log|F_{\zeta,N}'(z)|
&=
\frac{1}{2\pi}\int_0^{2\pi}
\log\left|
e^{i\theta}
+
\sum_{I\in\HD^{(N)}}\sum_{\nu=1}^{N_0}
\zeta_{I,\nu}\bigl(d_I^{(N)}\bigr)^{1/2}
B_bb_{I,\nu}(z)
\right|\,d\theta\\
&=
\log\max\left\{
1,
\left|
\sum_{I\in\HD^{(N)}}\sum_{\nu=1}^{N_0}
\zeta_{I,\nu}\bigl(d_I^{(N)}\bigr)^{1/2}
B_bb_{I,\nu}(z)
\right|
\right\}
\geq0.
\end{aligned}
$$
Averaging over the remaining Steinhaus variables proves
\eqref{20260728eq30}.

\medskip

Fix
$$
z\in\bigcup_{I\in\HD^{(N)}}Q_I^{\mathrm{up}}.
$$
Then $z\in Q_{I_z,\nu}\subseteq Q_{I_z}^{\mathrm{up}}$ for some
$I_z\in\HD^{(N)}$ and $1\leq\nu\leq N_0$. By Lemma~\ref{whitneyest},
$$
|B_bb_{I_z,\nu}(z)|^2
\gtrsim
\frac{1}{|I_z|^2K(|I_z|)}.
$$
Since
$$
d_{I_z}^{(N)}
=
|I_z|^{3/2}K(|I_z|),
$$
it follows that
$$
d_{I_z}^{(N)}|B_bb_{I_z,\nu}(z)|^2
\gtrsim
|I_z|^{3/2}K(|I_z|)
\frac{1}{|I_z|^2K(|I_z|)}
=
|I_z|^{-1/2}.
$$
Combining this estimate with \eqref{20260728eq04}, we obtain
\begin{align*}
\mathbb E_{\zeta}\log|F_{\zeta,N}'(z)|
&\geq
\frac12\log\left(
1+d_{I_z}^{(N)}|B_bb_{I_z,\nu}(z)|^2
\right)
-\frac12\log2\\
&\geq
\frac12\log\left(1+c_0|I_z|^{-1/2}\right)
-\frac12\log2\\
&\geq
c_1\log\frac{e}{|I_z|}-C_1.
\end{align*}
Since $1-|z|^2\simeq|I_z|$ on $Q_{I_z}^{\mathrm{up}}$, we conclude that
\begin{equation}\label{20260728eq31}
\mathbb E_{\zeta}\log|F_{\zeta,N}'(z)|
\geq
c_2\log\frac{e}{1-|z|^2}-C_2,
\qquad
z\in\bigcup_{I\in\HD^{(N)}}Q_I^{\mathrm{up}}.
\end{equation}
Here, the constants $c_0,c_1,c_2,C_1,C_2>0$ are independent of $N$.

\medskip

Finally, taking expectations on both sides of \eqref{20260728eq05}, applying
Fubini's theorem, and using \eqref{20260728eq30} and
\eqref{20260728eq31}, we obtain
\begin{align*}
W(\D)
&=
\mathbb E_{\zeta}W(\D)\\
&\gtrsim
\int_{\D}
W(z)\mathbb E_{\zeta}\log|F_{\zeta,N}'(z)|\,dA(z)\\
&\geq
\int_{\bigcup_{I\in\HD^{(N)}}Q_I^{\mathrm{up}}}
W(z)\mathbb E_{\zeta}\log|F_{\zeta,N}'(z)|\,dA(z)\\
&\geq
c_2
\int_{\bigcup_{I\in\HD^{(N)}}Q_I^{\mathrm{up}}}
W(z)\log\frac{e}{1-|z|^2}\,dA(z)
-C_2W(\D).
\end{align*}
Therefore,
$$
\sup_{N\geq0}
\int_{\bigcup_{I\in\HD^{(N)}}Q_I^{\mathrm{up}}}
W(z)\log\frac{e}{1-|z|^2}\,dA(z)
<\infty.
$$
Letting $N\to\infty$ proves {\rm (iii)} and completes the proof.
\end{proof}

\end{document}